\documentclass[a4paper]{amsart}
\usepackage{graphicx}
\usepackage{amssymb} 
\usepackage{stmaryrd}
\usepackage{dsfont}
\usepackage{colonequals}
\usepackage[mathcal]{euscript}
\usepackage{tikz}
\usepackage{tikz-cd}
\usepackage{hyperref}
\hypersetup{%
  bookmarksnumbered=true,%
  colorlinks=true,%
  linkcolor=blue,%
  citecolor=blue,%
  filecolor=blue,%
  menucolor=blue,%
  urlcolor=blue,%
  bookmarksopen=true,%
  bookmarksdepth=2,%
  pageanchor=true}

\title{An analogue of Stone duality via support}

\author{Henning Krause}
\address{Fakult\"at f\"ur Mathematik\\
Universit\"at Bielefeld\\ D-33501 Bielefeld\\ Germany}
\email{hkrause@math.uni-bielefeld.de}

\theoremstyle{plain}
\newcounter{thms}
\newtheorem{thm}[thms]{Theorem}
\newtheorem{prop}[thms]{Proposition}
\newtheorem{lem}[thms]{Lemma} 

\newtheorem{cor}[thms]{Corollary}

\theoremstyle{definition}
\newtheorem{defn}[thms]{Definition}
\newtheorem{exm}[thms]{Example}

\theoremstyle{remark}
\newtheorem{rem}[thms]{Remark}

\newcommand{\cl}{\operatorname{cl}}
\newcommand{\Cl}{\operatorname{Cl}}

\newcommand{\Ext}{\operatorname{Ext}}

\newcommand{\Hom}{\operatorname{Hom}}

\newcommand{\Id}{\operatorname{Id}}

\renewcommand{\mod}{\operatorname{mod}}

\newcommand{\Ob}{\operatorname{Ob}}

\newcommand{\Pt}{\operatorname{Pt}}

\newcommand{\Sp}{\operatorname{Sp}}
\newcommand{\Spc}{\operatorname{Spc}}
\newcommand{\Spec}{\operatorname{Spec}}

\newcommand{\supp}{\operatorname{supp}}

\newcommand{\Thick}{\operatorname{Thick}}

\newcommand{\BLat}{\mathbf{BLat}} 
\newcommand{\Frm}{\mathbf{Frm}} 
\newcommand{\JSLat}{\mathbf{JSLat}} 
\newcommand{\Top}{\mathbf{Top}} 
\newcommand{\two}{\mathbf{2}}

\newcommand{\op}{\mathrm{op}}

\newcommand{\iso}{\xrightarrow{\raisebox{-.4ex}[0ex][0ex]{$\scriptstyle{\sim}$}}}

\newcommand{\longiso}{\xrightarrow{\ \raisebox{-.4ex}[0ex][0ex]{$\scriptstyle{\sim}$}\ }}

\newcommand{\principal}[1]{\left\downarrow#1\right.}

\newcommand*{\intref}[2]{\def\tmp{#1}\ifx\tmp\empty\hyperref[#2]{\ref*{#2}}\else\hyperref[#2]{#1~\ref*{#2}}\fi}

\def\B{\mathcal B} 
\def\C{\mathcal C}

\def\calS{\mathcal S} 
\def\T{\mathcal T}

\def\bfD{\mathbf D}

\newcommand{\frp}{\mathfrak{p}}

\def\p{\phi}
\def\s{\sigma}
\def\t{\tau}

\def\Si{\Sigma}

\begin{document}



\begin{abstract}
  The notion of support provides an analogue of Stone duality,
  relating lattices to topological spaces. This note aims to explain in
  lattice theoretic terms what has been developed in the context of
  triangulated categories. In particular, the parallel between support
  via closed and open sets is addressed in terms of Hochster
  duality. As an application we indicate some consequences for tensor
  exact categories.
\end{abstract}

\date{July 23, 2023}

\maketitle

The traditional way of defining support assigns to an object a closed
subset of an appropriate topological space. On the other hand, there
is Stone duality which provides a correspondence between spaces and
their lattices of open subsets. This note aims to explain the connection
between both concepts which is based on Hochster duality.

Interest in this topic started with the seminal work of Balmer on
support for tensor triangulated categories \cite{Ba2005}. Soon after
the first paper appeared, several authors noticed the connection with
Stone duality via Hochster duality \cite{BKS2007,KP2017}. More recent
work focusses on support for triangulated categories, without assuming
a tensor product, or at least without assuming it to be symmetric
\cite{BO2023,GS2022,Kr2023,NVY2019}. This note is following these
ideas, without claiming any originality. The purpose is to treat the
subject purely in terms of lattices, to make some beautiful ideas
accessible for a wider audience. In particular, one may apply the
notion of support in other algebraic contexts, beyond triangulated
categories.  As an illustration we indicate the rudiments of `tensor
exact geometry'. A further bonus is a version of Stone duality that
does not require lattices to be distributive.

\section{Spaces}

We write $\two=\{0,1\}$ and view it either as a lattice via the usual
partial ordering or as topological space, where $\{1\}$ is open but
not closed. For a lattice $L$ let $L^\op$ denote its dual lattice. For
a topological space $X$ we write $\Omega(X)=\Hom_\Top(X,\two)$ for the
lattice of open sets and $\Cl(X)=\Hom_\Top(X,\two)^\op$ for the lattice of
closed sets.

\section{Join-semilattices}\label{se:semi-latt}

Let $L$ be a \emph{join-semilattice}. Thus $L$ is a poset in which all
finite joins exist, including the join over the empty set which is the
unique minimal element. A morphism of join-semilattices is a map that
preserves all finite joins. In particular, the partial order is
preserved since
\[a\le b\;\iff \; a\vee b=b\quad \text{for}\quad a,b\in
  L.\]
An \emph{ideal} of $L$ is a
subset $I\subseteq L$ that is closed under finite joins and such that
$a\le b$ in $L$ with $b\in I$ implies $a\in I$. Equivalently, $I$ is
of the form $\p^{-1}(0)$ for a morphism $\p\colon L\to\two$.  We write
$\Id(L)$ for the lattice of ideals in $L$, with partial order given by
inclusion. Let us turn $\Id(L)$ into a topological space which we
denote $\Sp(L)$.  Set
\[\supp(a)\colonequals\{I\in \Sp(L)\mid a\not\in I\}\quad \text{for}\quad a\in
  L.\] Note that $\supp(a\vee b)=\supp(a)\cup\supp(b)$ for $a,b\in L$,
and $\bigcap_{a\in L}\supp(a)=\varnothing$.  Thus sets of the form
$\supp(a)$ yield a basis of closed sets for a topology on $\Sp(L)$.
 
\begin{defn}
  A \emph{support datum} on $L$ is a pair $(X,\s)$ consisting of a
  topological space $X$ and a morphism $\s\colon L\to \Cl(X)$.
  A morphism of support
    data $(X,\s)\to(Y,\t)$ is a continuous map $f\colon X\to
  Y$ such that $\s=\Cl(f)\circ\t$.
\end{defn}

A support datum is nothing but a map that assigns to each  $a\in L$ a closed
subset $\s(a)\subseteq X$ such that
  \begin{enumerate}
  \item[($\varnothing$)] $\s(0)=\varnothing$,
  \item[($\vee$)] $\s(a\vee b)=\s(a)\cup\s(b)$ for all $a,b\in L$.
  \end{enumerate}

  Predecessors of the following result are \cite[Theorem~2]{BO2023}
  and \cite[Proposition~5.4.2]{GS2022} in the context of triangulated
  categories.

  \begin{thm}\label{th:semilat}
    The functor $X\mapsto \Cl(X)$ from topological spaces to
    join-semilattices admits $L\mapsto \Sp(L)$ as a right
    adjoint. Thus there is a natural bijection
  \begin{equation}\label{semi-lat}
    \Si\colon\Hom_\Top(X,\Sp(L))\longiso\Hom_\JSLat(L,\Cl(X))
  \end{equation}
which takes a
  continuous map $f\colon X\to \Sp(L)$ to the support datum 
  \[a\longmapsto f^{-1}(\supp(a))\quad\text{for}\quad a\in L.\]  In particular, the pair
  $(\Sp(L),\supp)$ is the final support datum on $L$.
\end{thm}

\begin{proof}
  It is easily checked that the map $\Si$ is well defined; it takes
  the identity $\Sp(L)\to \Sp(L)$  to the support datum
  $(\Sp(L),\supp)$.  Given any support datum $(X,\s)$, we define
  $f\colon X\to \Sp(L)$ by setting
  \[f(x)\colonequals\{a\in L\mid x\not\in\s(a)\} \quad \text{for}\quad x\in X.\]
  Then  we have
  \[f^{-1}(\supp(a))=\{x\in X\mid a\not\in f(x)\}=\s(a) \quad
    \text{for}\quad a\in L.\] Thus $f$ is continuous and we see that
  $\Si$ is surjective. For the injectivity of $\Si$, observe that for any
  map  $f\colon X\to \Sp(L)$ and $x\in X$ we have \[f(x)=\{a\in L\mid
    a\in f(x)\}=\{a\in L\mid x\not\in f^{-1}(\supp(a))\}.\]
  Thus $f$ is determined by  $\Si(f)$.
\end{proof}

\begin{rem}
  (1) The assignment $\p\mapsto\p^{-1}(1)$ identifies $\Hom_\JSLat(L,\two^\op)$
with the ideal lattice $\Id(L)$. Taking $a\in L$ to the principal ideal
$\{b \in L\mid b \le a\}$ identifies $L$ with the
poset of \emph{compact elements} in $\Id(L)$.

(2) The join-semilattice $L$ can be
  recovered from the space $\Sp(L)$ as follows. The \emph{specialisation order}
\[x\leq y \;:\iff\; x\in\cl\{y\}\qquad (x,y\in\Sp(L))\]
recovers the partial order on $\Id(L)=\Sp(L)$ that is given by the
inclusion of ideals. Taking the lattice $\Id(L)$ to its poset of
compact elements yields $L$.

(3) The assignment $\p\mapsto\p^{-1}(0)$ identifies $\Hom_\JSLat(L,\two)$
with $\Sp(L)$ as sets. This yields the identification
\[\supp(a)=\{\p\in\Hom_\JSLat(L,\two)\mid \p(a)=1\}\quad \text{for}\quad a\in
  L.\] 
The choice of the topology (by declaring $\supp(a)$ to be closed)  suggests to write
$\Sp(L)=\Hom_\JSLat(L,\two)^\vee$.  Then one may rewrite the bijection
\eqref{semi-lat} as
\[\Hom_\Top(X,\Hom_\JSLat(L,\two)^\vee)\simeq\Hom_\JSLat(L,\Hom_\Top(X,\two)^\op).\]
\end{rem}

\section{Bounded lattices}\label{se:latt}

Let $L$ be a \emph{bounded lattice}. Thus $L$ is a poset in which all
finite joins and all finite meets exist. A morphism of bounded
lattices is a map that preserves all finite joins and all finite
meets. An ideal $I\subseteq L$ is \emph{prime} if $I\neq L$ and
$a\wedge b\in I$ implies $a\in I$ or $b\in I$. Equivalently, $I$ is of
the form $\p^{-1}(0)$ for a morphism $\p\colon L\to\two$.  We write
$\Spc(L)$ for the set of prime ideals of $L$ and turn this into a
topological space.  Set
\[\supp(a)\colonequals\{I\in \Spc(L)\mid a\not\in I\}\quad \text{for}\quad a\in
  L.\] Note that $\supp(a\vee b)=\supp(a)\cup\supp(b)$ for 
$a,b\in L$, and $\bigcap_{a\in L}\supp(a)=\varnothing$.  Thus sets
of the form $\supp(a)$ yield a basis of closed sets for a topology on
$\Spc(L)$.

\begin{defn}
  A \emph{closed support datum} on $L$ is a pair $(X,\s)$ consisting of a
  topological space $X$ and a morphism $\s\colon L\to \Cl(X)$.
  A morphism of closed support
    data $(X,\s)\to(Y,\t)$ is a continuous map $f\colon X\to
  Y$ such that $\s=\Cl(f)\circ\t$.
\end{defn}

A closed support datum is nothing but a map that assigns to each  $a\in L$ a closed
subset $\s(a)\subseteq X$ such that
  \begin{enumerate}
  \item[($\varnothing$)] $\s(0)=\varnothing$ and  $\s(1)=X$,
  \item[($\vee$)] $\s(a\vee b)=\s(a)\cup\s(b)$ for all $a,b\in L$,
  \item[($\wedge$)] $\s(a\wedge b)=\s(a)\cap\s(b)$ for all $a,b\in L$. 
  \end{enumerate}

The predecessor of the following result is
  \cite[Theorem~3.2]{Ba2005} in the context of tensor triangulated categories.
  
  \begin{thm}\label{th:lat}
    The functor $X\mapsto \Cl(X)$ from topological spaces to bounded
    lattices admits $L\mapsto \Spc(L)$ as a right adjoint. Thus there
    is a natural bijection
  \begin{equation*}\label{lat}
    \Hom_\Top(X,\Spc(L))\longiso\Hom_\BLat(L,\Cl(X))
\end{equation*}    
which takes a continuous map $f\colon X\to \Spc(L)$ to the closed
support datum
  \[a\longmapsto f^{-1}(\supp(a))\quad\text{for}\quad a\in L.\]  In particular, the pair
  $(\Spc(L),\supp)$ is the final support datum on $L$.
\end{thm}

\begin{proof}
 The proof of Theorem~\ref{th:semilat} carries over without any
 change. The additional structure on $L$ (given by the join $\wedge$)
 corresponds to the additional condition on the ideals in $\Spc(L)$
 (to be prime).
\end{proof}

\section{Hochster duality}

Let $L$ be a bounded lattice. We consider the space $\Spc(L)$ of prime
ideals and provide a dual topology on this set.  Observe that
$\supp(a\wedge b)=\supp(a)\cap\supp(b)$ for all $a,b\in L$, and
$\bigcup_{a\in L}\supp(a)=\Spc(L)$.  Thus sets of the form
$\supp(a)$ yield a basis of open sets for a topology on $\Spc(L)$. We
denote this space $\Spc(L)^\vee$ and call it the \emph{Hochster
  dual} of $\Spc(L)$. Note that \[\Spc(L)^\vee\cong\Spc(L^\op)\]
since \[\Hom_\BLat(L,\two^\op)=\Hom_\BLat(L^\op,\two).\]

\begin{defn}
  An \emph{open support datum} on $L$ is a pair $(X,\s)$ consisting of a
  topological space $X$ and a morphism $\s\colon L\to \Omega(X)$.
  A morphism of open support
    data $(X,\s)\to(Y,\t)$ is a continuous map $f\colon X\to
  Y$ such that $\s=\Omega(f)\circ\t$.
\end{defn}

There is a natural bijection
\begin{align*}\label{eq:dual}
  \Hom_\BLat(L,\Omega(X))\longiso\Hom_\BLat(L^\op,\Cl(X))
\end{align*}
that takes an open support datum $(X,\s)$  to the closed
  support datum $(X,\t)$ with
\[\t(a)\colonequals X\setminus\s(a)\quad\text{for}\quad a\in L^\op.\]
This yields the following  reformulation of Theorem~\ref{th:lat}.

\begin{thm}\label{th:distr-lat}\pushQED{\qed}
The functor $X\mapsto \Omega(X)$
from topological spaces to bounded lattices admits
$L\mapsto \Spc(L)^\vee$ as a right adjoint. Thus there is a natural bijection
\begin{equation}\label{distr-lat}
  \Hom_\Top(X,\Spc(L)^\vee)\longiso\Hom_\BLat(L,\Omega(X)).\qedhere
\end{equation}
\end{thm}

\begin{rem}
(1)  The usual setting for Hochster duality are \emph{spectral
  spaces}, so spaces  of the form $\Spc(L)$ for some
  bounded distributive lattice $L$ \cite[II.3.4]{Jo1982}.

(2)  The adjunction formula \eqref{distr-lat}  may be rewritten as
\[  \Hom_\Top(X,\Hom_\BLat(L,\two))\simeq\Hom_\BLat(L,\Hom_\Top(X,\two)).\]
\end{rem}

\section{Frames}

We recall some well known facts from Stone duality.
A \emph{frame} is a complete lattice in which finite meets distribute
over arbitrary joins, so
\[a\wedge\big(\bigvee_i b_i\big)=\bigvee_i(a\wedge b_i)\quad\text{for all}\quad
  a,b_i\in F.\]
A morphism of frames is a map that preserves
all joins and finite meeets. The functor
sending a topological space $X$ to its frame $\Omega(X)$ of open sets
admits a right adjoint, which sends a frame $F$ to its space $\Pt(F)$
of points \cite[II.1.4]{Jo1982}. A \emph{point} of $F$ is by definition a frame morphism
$F\to\two$, and an open set is one of the form
\[U(a)=\{\p\in\Pt(F)\mid \p(a)=1\}\quad\text{for some}\quad a\in F.\]
This adjunction amounts to a
natural bijection
\begin{equation}\label{stone}
  \Hom_\Top(X,\Pt(F))\longiso\Hom_\Frm(F,\Omega(X)).
\end{equation}
A frame $F$ is \emph{spatial}
if there are enough points, which means that the unit of the
adjunction (given by evaluation) yields an isomorphism
\begin{equation}\label{spatial}
F\longiso \Omega(\Pt(F)).
\end{equation}

Let $L$ be a bounded lattice that is distributive. Then its ideal
lattice $\Id(L)$ is a spatial frame \cite[II.3.4]{Jo1982}. A frame
isomorphic to one of the form $\Id(L)$ is called \emph{coherent}.  Let
us consider the embedding $L\to\Id(L)$ which takes $a\in L$ to the
principal ideal
\[\principal{a}\colonequals\{b \in L\mid b \le a\}.\]

\begin{lem}
Restriction along $L\to\Id(L)$ induces for any frame $F$ a natural bijection
\begin{equation}\label{Hom}
  \Hom_\Frm(\Id(L),F)\longiso\Hom_\BLat(L,F).
\end{equation}
\end{lem}
\begin{proof}
  The inverse map takes a morphism $\phi\colon L\to F$ to $\Id(L)\to F$ given by
  \[ I\longmapsto\bigvee_{a\in I}\phi(a)\quad\text{for}\quad
    I\in\Id(L).\qedhere\]
\end{proof}

Take $F=\two$. Then the above bijection identifies each point
$\p\in\Pt(\Id(L))$ with the prime ideal $\p^{-1}(0)\cap L$ of $L$, and
this yields an isomorphism
\begin{equation}\label{Hochster}
   \Pt(\Id(L))\longiso\Spc(L)^\vee
 \end{equation}
between the space of points of $\Id(L)$ and the Hochster dual
 of $\Spc(L)$. More precisely, for any ideal $I\in\Id(L)$ we have
 \[U(I)=\bigcup_{a\in I}U(\principal{a})\]
and \eqref{Hochster} identifies  $U(\principal{a})$ with $\supp(a)$
for each $a\in L$.

\begin{cor}\label{co:Id(L)}
Let $L$ be a bounded lattice that  is distributive. Then the assignment
\[I\longmapsto \bigcup_{a\in I}\supp(a)\]
yields an isomorphism
\begin{equation}\label{ideals}
  \Id(L)\longiso\Omega(\Spc(L)^\vee).
\end{equation}  
\end{cor}
\begin{proof}
  The isomorphism is a consequence of \eqref{spatial} since $\Id(L)$
  is a spatial frame; its inverse sends an open set $U$ to
  $\{a\in L\mid \supp(a)\subseteq U\}$. 
\end{proof}

\begin{rem}
(1)  Stone duality yields an alternative proof of
  Theorem~\ref{th:distr-lat} when the lattice $L$ is distributive,
  by combining \eqref{stone}, \eqref{Hom}, and \eqref{Hochster}.

(2)  The adjunction formula \eqref{stone}  may be rewritten as
\[  \Hom_\Top(X,\Hom_\Frm(F,\two))\simeq\Hom_\Frm(F,\Hom_\Top(X,\two)).\]

(3) With the canonical identification $\Id(L)=\Hom_\JSLat(L,\two^\op)$ for
  a join-semilattice $L$, the isomorphism \eqref{Hochster}  may be rewritten as
  \[ \Hom_\Frm(\Hom_\JSLat(L,\two^\op),\two)\longiso\Hom_\BLat(L,\two)\] whereas the
  isomorphism \eqref{ideals} becomes
  \[  \Hom_\JSLat(L,\two^\op) \longiso\Hom_\Top(\Hom_\BLat(L,\two),\two).\]
\end{rem}

\section{Triangulated categories and beyond}

Let $\T$ be an essentially small triangulated category and let
$\Ob\T$ denote the class of objects. For objects $X,Y$ in $\T$ we write
\[X\sim Y \;:\iff\; [X]=[Y]\] where $[ X]$ denotes
the thick subcategory generated by $X$.  This provides an equivalence
relation and the set of equivalence classes
\[L(\T)\colonequals\Ob\T/{\sim}\] is partially ordered via inclusion;  it is a
join-semilattice with
\[ [X]\vee [Y]=[X \oplus Y].\] 
Let   $\Thick(\T)$ denote the lattice of thick subcategories of $\T$.

\begin{lem}\label{le:thick}
  The assignment
  \[\T\supseteq\calS\longmapsto\{[X]\in L(\T)\mid X\in\calS\}\subseteq
    L(\T)\]
  yields a lattice isomorphism $\Thick(\T)\iso\Id(L(\T))$.
\end{lem}
\begin{proof}
  The inverse map sends an ideal $I\subseteq L(\T)$ to the full
  subcategory of $\T$ that is given by the objects $X$ with $[X]\in
  I$.
\end{proof}

Now let $(\T,\otimes)$ be a tensor triangulated category, i.e.\ a
triangulated category equipped with a monoidal structure (not
necessarily symmetric) which is exact in each variable.  A thick
subcategory $\calS\subseteq\T$ is a \emph{tensor ideal} if
$X\otimes Y$ is in $\calS$ provided that one of $X,Y$ is in $\calS$,
and an ideal is \emph{radical} if $X$ is in $\calS$ provided that
$X\otimes X$ is in $\calS$.

For objects $X,Y$ in $\T$ we write
\[X\approx Y \;:\iff\; \langle X\rangle=\langle Y\rangle\] where
$\langle X\rangle$ denotes the radical thick tensor ideal generated by
$X$. We obtain an equivalence relation and the set of
equivalence classes
\[L(\T,\otimes)\colonequals\Ob\T/{\approx}\] is partially ordered
via inclusion.

\begin{lem}\label{le:tensor}
For objects $X,Y$ in $\T$ we have $\langle X\rangle\cap \langle
Y\rangle=\langle X\otimes Y\rangle$.
\end{lem}
\begin{proof}
One inclusion is clear. Let $Z \in\langle X\rangle\cap \langle
Y\rangle$. Then we compute
\[\langle Z\rangle\subseteq \langle Z\otimes Z\rangle\subseteq
  \langle Z\otimes Y\rangle\subseteq \langle X\otimes Y\rangle.\]
Thus $Z\in \langle X\otimes Y\rangle$.
\end{proof}

We record the basic properties of $L(\T,\otimes)$.

\begin{prop}\label{pr:L(T)}
  Let   $(\T,\otimes)$ be a tensor triangulated category.
Then the poset  $L(\T,\otimes)$ is a distributive lattice and its ideal lattice
$\Id(L(\T,\otimes))$ identifies with the lattice of radical thick tensor
ideals of $\T$.
\end{prop}
\begin{proof}
  For objects $X,Y$ in $\T$ we have 
  \[ \langle X\rangle\vee \langle Y\rangle=\langle X \oplus
    Y\rangle\qquad\text{and}\qquad \langle X\rangle\wedge \langle Y
    \rangle=\langle X \otimes Y\rangle\] thanks to
  Lemma~\ref{le:tensor}.  The tensor product distributes over sums and
  this implies the distributivity of $L(\T,\otimes)$.  The second
  assertion is the analogue of Lemma~\ref{le:thick} and its proof
  carries over without change.
\end{proof}

Let us mention a couple of consequences. For example, Corollary~\ref{co:Id(L)} appplies and yields a
description of the lattice of radical thick tensor ideals of $\T$ in
terms of the space $\Spc(L(\T,\otimes))$
\cite[Theorem~10.4]{Ba2005}. Also, we see that the lattice of radical
thick tensor ideals of $\T$ is a coherent frame
\cite[Theorem~3.1.9]{KP2017}.

We conclude that the material developed in the previous sections can
be applied towards the description of thick subcategories and thick
tensor ideals. For some recent applications of
this lattice theoretic approach, see for example \cite{BCHNPS2023}.

The thoughtful reader will notice that other settings (different from
triangulated categories) are perfectly feasible. Classical examples
are categories of modules or sheaves and their Serre subcategories;
this is in line with Gabriel's reconstruction of a noetherian scheme
from its category of coherent sheaves \cite{Ga1962}. 

\section{Tensor exact geometry}

For exact categories we indicate the rudiments of `tensor exact
geometry' which may be viewed as the analogue of tensor triangular
geometry \cite{Ba2010}. Let $(\C,\otimes)$ be a \emph{tensor exact category}; thus
$\C$ is an exact category in the sense of Quillen (so equipped with a
distinguished class of exact sequences) and there is a monoidal
structure (not necessarily symmetric) which is exact in each
variable. A \emph{thick subcategory} is a full additive subcategory
$\B\subseteq\C$ satisfying the \emph{two-out-of-three property}, so
any exact sequence from $\C$ lies in $\B$ if two of its three terms
are in $\B$. We make the same definitions as before for a tensor
triangulated category and obtain with same proof the analogue of
Proposition~\ref{pr:L(T)}.

\begin{prop}
   Let   $(\C,\otimes)$ be a tensor exact category.
Then the poset  $L(\C,\otimes)$ is a distributive lattice and its ideal lattice
$\Id(L(\C,\otimes))$ identifies with the lattice of radical thick tensor
ideals of $\C$.\qed
\end{prop}

A cornerstone of tensor triangular geometry is the classification of radical thick tensor
ideals \cite[Theorem~10.4]{Ba2005}; its analogue for tensor exact
categories is now a consequence of Corollary~\ref{co:Id(L)}.

\begin{cor}
  Taking an object to its  support yields an isomorphism between  the lattice of radical thick tensor
ideals of $\C$ and $\Omega(\Spc(L(\C,\otimes))^\vee)$.\qed
\end{cor}

\begin{exm}
Let $G$ be a finite group scheme over a field $k$. We consider
the category  $(\C,\otimes)=(\mod kG,\otimes)$ of finite dimensional $k$-linear
representations of $G$ with the usual tensor exact structure. We write
$H^*(G,k)$ for the cohomology ring of $G$ with coefficients in $k$
and $\Spec(H^*(G,k))$ for its spectrum of homogeneous prime ideals
(with the Zariski topology). Then results from
\cite{Benson/Iyengar/Krause/Pevtsova:2018, Friedlander/Pevtsova:2007a}
yield a homeomorphism
\[\Spec(H^*(G,k)) \longiso \Spc(L(\mod kG,\otimes))\]
given by
\[H^*(G,k)\supseteq \frp\longmapsto \{M\in\mod
  kG\mid \Ext_{kG}^*(M,M)_\frp=0\}\subseteq L(\mod kG,\otimes).\]
Let $\bfD^b(\mod kG)$ denote the bounded derived category of $\mod
kG$ with the induced tensor structure. It is interesting to note that
the inclusion $\mod kG\to\bfD^b(\mod kG)$ induces an isomorphism
\[L(\bfD^b(\mod kG),\otimes)\longiso L(\mod kG,\otimes),\]
reconciling triangulated and exact structure.
\end{exm}

\subsection*{Acknowledgements}

It is a pleasure to thank Greg Stevenson for several helpful
discussions. Also, I am grateful to Paul Balmer for sharing some
private notes \cite{BO2023}. This work was supported by the Deutsche
Forschungsgemeinschaft (SFB-TRR 358/1 2023 - 491392403).


\begin{thebibliography}{10}

\bibitem{Ba2005} P. Balmer, \emph{The spectrum of prime ideals in
    tensor triangulated categories}, J. Reine
  Angew. Math. \textbf{588} (2005), 149--168.

\bibitem{Ba2010} P. Balmer, \emph{Tensor triangular geometry}, in
  {Proceedings of the International Congress of Mathematicians.
    Volume II}, 85--112, Hindustan Book Agency, New Delhi, 2010.

\bibitem{BO2023} P. Balmer, P. S. Ocal, \emph{Universal support for
    triangulated categories}, private communication, 2023.

\bibitem{BCHNPS2023} T. Barthel, N. Castellana, D. Heard, N. Naumann, L. Pol,
  and B. Sanders, Descent in tensor triangular geometry, arXiv:2305.02308.

\bibitem{Benson/Iyengar/Krause/Pevtsova:2018} D.~J. Benson,
S.~B. Iyengar, H.~Krause, and J.~Pevtsova, \emph{Stratification for
  module categories of finite group schemes},
J.~Amer.~Math.~Soc. \textbf{31} (2018), no.~1, 265--302.

\bibitem{BKS2007} A. B. Buan, H. Krause, and \O{}. Solberg, \emph{Support
  varieties–an ideal approach}, Homology, Homotopy Appl., 9 (2007),
  45–74.

\bibitem{Friedlander/Pevtsova:2007a}
E.~M. Friedlander and J.~Pevtsova, \emph{{$\Pi$-supports for modules for finite groups schemes}}, Duke
  Math.\ J. \textbf{139} (2007), 317--368.

\bibitem{Ga1962} P. Gabriel, \emph{Des cat\'egories ab\'eliennes},
  Bull. Soc. Math. France {\bf 90} (1962), 323--448.
  
\bibitem{GS2022} S. Gratz and G. Stevenson, \emph{Approximating triangulated
    categories by spaces}, arXiv:2205.13356.

\bibitem{Jo1982} P. T. Johnstone, \emph{Stone spaces}, Cambridge
  Stud. Adv. Math., vol.~3, Cambridge Univ. Press, Cambridge, 1982.

\bibitem{KP2017} J. Kock\ and\ W. Pitsch, \emph{Hochster duality in derived
  categories and point-free reconstruction of schemes},
Trans. Amer. Math. Soc. {\bf 369} (2017), no.~1, 223--261.

\bibitem{Kr2023} H. Krause, \emph{Central support for triangulated
  categories}, arXiv:2301.10464.

\bibitem{NVY2019} D. K. Nakano, K. B. Vashaw\ and\ M. T. Yakimov,
  \emph{Noncommutative tensor triangular geometry}, Amer. J. Math. {\bf 144}
  (2022), no.~6, 1681--1724.
   
\end{thebibliography}
\end{document}